\documentclass[12pt]{article}
\usepackage{amsmath,amssymb}
\usepackage{amsthm}

\def\LC{\mathcal{L}}

\def\FC{\mathcal{F}}
\def\HC{\mathcal{H}}

\def\E{\mathbf{E}}

\def\P{\mathbf{P}}

\def\1{\mathbf{1}}

\def\e{\mathbf{e}}

\def\tr{\rm{tr}}

\def\pa{\partial}

\def\de{\delta}
\def\ga{\gamma}

\newtheorem{prop}{Proposition}[section]
\newtheorem{theorem}{Theorem}[section]
\newtheorem{lemma}{Lemma}[section]

\newtheorem{remark}{Remark}

\newcommand{\si}{\sigma}

\begin{document}
\title{On quantum Stochastic Master equations
}

\author{Vassili N. Kolokoltsov\thanks{Faculty of Computation Mathematics and Cybernetics,
 Moscow State University, 119991 Moscow, Russia, Higher School of Economics, Russia,
 Professor Emeritus of the University of Warwick,
  Email: v.n.kolokoltsov@gmail.com and v.kolokoltsov@warwick.ac.uk}}
\maketitle

\begin{abstract}
Stochastic Master equations or quantum filtering equations for mixed states are well known objects in quantum physics.
Building a mathematically rigorous theory of these equations in infinite-dimensional spaces is a long standing open problem.
The first objective of this paper is to give a solution to this problem under the assumption
 of bounded operators providing coupling with environment (or a measurement devise). Furthermore,
 recently the author built the theory of the law of large number limit for continuously observed
 interacting quantum particle systems leading to quantum mean-field games. These limits are described
 by certain nontrivial extensions of quantum stochastic master equations that can be looked at as
 infinite-dimensional operator-valued McKean-Vlasov diffusions. The second objective of this paper is
 to provide a well-posedness result for these new class of McKean-Vlasov diffusions.
\end{abstract}

{\bf Key words:} quantum stochastic master equation, stochastic Lindblad equation,
quantum stochastic filtering, unravelling of quantum dynamic semigroups, quantum trajectories,
quantum interacting particle systems, quantum law of large numbers, singular SDEs in Banach spaces,
infinite-dimensional McKean-Vlasov diffusions.

\section{Preliminaries: filtering equations for pure states}

The general theory of quantum non-demolition observation, filtering and resulting
feedback control was built essentially in papers  \cite{Bel87}, \cite{Bel88}, \cite{Bel92}.
A well written review of this theory can be found in \cite{BoutHanJamQuantFilt}.
For an alternative simplified derivations of the main filtering equations
(by-passing the heavy theory of quantum filtering) we refer to \cite{BelKol},
\cite{Pellegrini}, \cite{BarchBel}, \cite{Holevo91}, \cite{KolQuantFrac} and references therein.
For the technical side of organising feedback quantum control in real time,
see e.g. \cite{Armen02Adaptive}, \cite{Bushev06Adaptive} and \cite{WiMilburnBook}.
Equations of quantum filtering can be also looked at as stochastic master (or Lindblad)
 equation yielding unravelling of quantum dynamic semigroups,
see various kind of interpretations and lots of references in monograph \cite{BarchBook}.

Continuous measurement and quantum filtering of quantum systems can be organised in two versions:
counting and diffusive type detection. In this paper we shall deal only with the latter.
The corresponding dynamics can be described in terms of the stochastic evolution of pure or mixed states.
The mathematics of the evolution of pure states given by the Belavkin equations (and representing some kind
of stochastic nonlinear Schr\"odinger equation) is fairly well understood by now. The present paper
is devoted to the study of a more subtle case of the operator-valued evolution of mixed states.
 Moreover,  recently the author built the theory of the law of large number limit for continuously observed
 interacting quantum particle systems leading to quantum mean-field games, see \cite{KolQuantLLN},
 \cite{KolQuantMFGCount} and \cite{KolQuantMFG}. These limits are described
 by certain nontrivial extensions of quantum stochastic master equations that can be looked at as
 infinite-dimensional operator-valued McKean-Vlasov diffusions. The second objective of this paper is
 to provide a well-posedness result for these new class of McKean-Vlasov diffusions.

Let us introduce basic notations and write down the main equations in order
to better explain our findings. For derivations of these equations see references given above.

To speak in a unified way about the so-called output and innovation processes
(that are Brownian motions under different measures) we shall use the common
notion of an Ito process $X(t)$ defined on some filtered probability space
$(\Omega, \FC, \FC_t, \P)$ and given by its stochastic differential
$dX(t)=A(t) dt+\si(t)\, dW(t)$, where $W(t)=(W_1, \cdots, W_n)(t)$ is a standard
$n$-dimensional $\FC_t$-Wiener process (or Brownian motion) and $A(t), \si(t)$ (where $A(t)$
a vector and $\si(t)$ a matrix) are continuous adapted processes
on $(\Omega, \FC, \FC_t, \P)$.  In this particular study it will be convenient
(for more general and at the same time unified formulations)
to use a nonstandard notion of a {\it simple Ito process} to indicate an Ito process
starting from zero with bounded $A(t)$ and with the unit matrix $\si(t)$,
so that it satisfies the Ito product rule in the form $dX_tdX_t=dt$
(in more detailed notation $dX_i(t) dX_j(t)=\de^i_j dt$).

\begin{remark}
By Girsanov's theorem, all simple processes can be turned to a Brownian motion
by a change to an equivalent measure. This fact is crucial for
physical interpretations, but will be used by us only occasionally
and even can be avoided altogether.
\end{remark}

In this paper letters $H$ and $L=(L_1, \cdots, L_n)$ will denote linear operators
in a separable Hilbert space $\HC$. Here $H$ is self-adjoint and referred to as
the Hamiltonian. The vector-valued $L$, assumed to be bounded, describes
the coupling operator with the measurement device. We shall use the notation
$L_S=(L+L^*)/2=(L_{S1}, \cdots, L_{Sn})$ and $L_A=(L-L^*)/2i=(L_{A1}, \cdots, L_{An})$
for symmetric and antisymmetric parts of $L$ with the adjoint $L^*$, and $\|L\|=\sum_j \|L_j\|$.
We shall write $\, Re \, z$ and $Im \, z$ for the real and imaginary parts of a complex number
or vector. The brackets $[A,B]$ and $\{A,B\}$
will denote the commutator and anti-commutator of operators $A,B$. By $\|A\|$ we denote
the standard operator norm of a bounded operator$A$ in $\HC$. The norms in other spaces of operators,
like spaces of Hilbert-Schmidt operators, will be marked by lower index indicating the corresponding space.

\begin{remark}
Everything given below have a straightforward extension to the case of time dependent families $L(t)$, as
long as these families are measurable and uniformly bounded. In fact, time dependent families arise necessarily
when reducing the case of unbounded $H$ to the case of bounded (in fact, vanishing) $H$ via the interaction
 representation, see below. If $L$ is time-dependent, all estimates below are valid with $\|L\|=\max_t\|L(t)\|$.
Similarly the theory extends to time-dependent Hamiltonian $H_t$ whenever such family generates a well defined
unitary propagator.
\end{remark}

The quantum filtering equation  describing the stochastic evolution of pure states
(vectors in a Hilbert space $\HC$) under continuous measurement
of a diffusive type can be written in two equivalent ways:

(i) as the {\it linear Belavkin quantum filtering equation} for a non-normalized state:
\begin{equation}
\label{eqqufiBlin}
d\chi(t) =-[iH\chi(t) +\frac12 \sum_j L_j^*L_j \chi(t) ]\,dt+\sum_jL_j\chi(t) dY_j(t),
\end{equation}
where $\chi(t)\in \HC$ and $Y(t)$
is a simple $n$-dimensional Ito process referred to as the {\it output process};
for bounded $H,L$, equation \eqref{eqqufiBlin} is clearly well-posed, as a standard linear Ito's equation
in a Hilbert space with bounded coefficients;

(ii) as the {\it nonlinear Belavkin quantum filtering equation}
 for the normalized state $\phi(t)$:
\[
d\phi(t)=\sum_j(L_j-(\phi(t), L_{Sj} \phi(t)))\phi(t) \, dB_j(t)
\]
\begin{equation}
\label{eqqufiBnonlin}
-[i(H-\sum_j(\phi(t), L_{Sj} \phi(t)) L_{Aj})
+\frac12\sum_j (L_j-(\phi(t), L_{Sj} \phi(t)))^*(L-(\phi(t), L_{Sj} \phi(t)))]\phi(t) \, dt,
\end{equation}
where $B_t$ is a simple $n$-dimensional Ito process, referred to as the {\it innovation process}.

The link between these equations will be explained below.

\begin{remark} The square norm $\|\chi\|^2$ of solutions to \eqref{eqqufiBlin} has physical meaning
analogous to the square norm of the solutions to the standard Schr\"odinger equation:
it describes the probability density of observing corresponding values of the output process $Y(t)$,
see detailed discussion in  \cite{BarchBook}.
\end{remark}

If one agrees to understand all products of operator expressions as appropriate inner products
(sum over available indices), for instance, writing $L^*L$ instead of $\sum_j L_j^* L_j$,
and $L \chi \, dY(t)$ instead of $\sum_j L_j\chi \, dY_j(t)$,
 one can write all equations above in a simpler form,
say the main equations \eqref{eqqufiBlin} and \eqref{eqqufiBnonlin} will look like
\begin{equation}
\label{eqqufiBlins}
d\chi(t) =-[iH\chi(t) +\frac12 L^*L \chi(t) ]\,dt+L \chi(t) dY(t),
\end{equation}
and, respectively,
\[
d\phi(t)=-[i(H-(\phi(t), L_S \phi(t)) L_A)
+\frac12 (L-(\phi(t), L_S \phi(t)))^*(L-(\phi(t), L_S \phi(t)))]\phi(t) \, dt
\]
\begin{equation}
\label{eqqufiBnonlins}
+ (L-(\phi(t), L_S \phi(t)))\phi(t) \, dB(t).
\end{equation}
We will mostly use this short way of writing having in mind more detailed versions above.

Recall that the density matrix or density operator $\ga$ corresponding to a unit vector
$\chi\in \HC$ is defined as the orthogonal projection operator on $\chi$.
This operator is usually expressed either as the tensor product  $\ga=\chi\otimes \bar \chi$
(with the usual identification of $H\otimes H$ with the space of Hilbert-Schmidt operators in $\HC$)
or in the most common for physics bra-ket Dirac's notation as $\ga=|\chi\rangle \langle \chi|$.

As one checks by direct application of Ito's formula, (i) if $\chi(t)\in \HC$
satisfies \eqref{eqqufiBlin}, then the corresponding operator $\ga=\chi\otimes \bar \chi$
satisfies the {\it linear stochastic quantum master equation}
\begin{equation}
\label{Lindstoch}
d\ga(t)=-i[H,\ga(t)] \, dt +\LC_L \ga(t) \, dt +(L\ga(t)+\ga(t) L^*) dY(t),
\end{equation}
with
\[
\LC_L\ga =L\ga L^*-\frac12 L^*L\ga -\frac12 \ga L^*L
=L\ga L^*-\frac12 \{L^*L,\ga\};
\]
and (ii) if $\phi(t)$ satisfies \eqref{eqqufiBnonlin}, then the corresponding matrices
$\rho=\phi\otimes \bar \phi$ satisfies the {\it nonlinear stochastic quantum master equation}
\begin{equation}
\label{Lindstochnorm1}
d\rho(t)=-i[H,\rho(t)]\, dt+\LC_L \rho (t)\, dt
+[L\rho(t)+\rho(t) L^*-\rho(t)\, {\tr} \, (L\rho(t)+\rho(t) L^*) ] dB(t).
\end{equation}

These master equations and their nonlinear extensions for interacting particle systems are the main
objects of study in this paper. They describe quantum evolutions under continuous observation
in terms of mixed states (positive operators of unit trace) not necessarily arising from
pure states, that is, not necessarily of form $\chi \otimes \bar \chi$.

Our first main result concern the well-posedness of equation \eqref{Lindstochnorm1}
(or its mild version in case of unbounded $H$) for bounded operator $L$. This includes existence and
uniqueness of solutions and continuous dependence on parameter such as operator $H$. As a tool for proving it,
we show that this equation can be rewritten in an equivalent way as an equation for pure states in an
appropriately chosen Hilbert space.  Secondly, we extend this result to more general equations providing
the law of large numbers for continuously observed interacting quantum particle systems.

The paper is organised a follows.

Section \ref{secpureeq} is a warm-up, where we collect some auxiliary facts
on the equations for pure states. Results here are mostly known, but we stress certain details
that are important for further development.

Section \ref{seclineq} is devoted to the linear SDE \eqref{Lindstoch}. It makes an important first
step for tackling nonlinear equation \eqref{Lindstochnorm1}. The main point here is to prove
the preservation of positivity of the solutions and some bounds for their traces.

Section \ref{secnormeq} presents our first main result. We show the well-posedness
of equation \eqref{Lindstochnorm1} (or its mild extension in case of unbounded $H$),
including continuous dependence on the Hamiltonian $H$.

In Section \ref{secMFeq} our second main result is obtained. We prove well-posesdness
for infinite-dimensional operator-valued McKean-Vlasov diffusions obtained as the law
of large number limit of continuously observed quantum particle systems. These equations are obtained from
quantum stochastic master equations for open systems in the way that is analogous to obtaining
standard nonlinear Schr\"odinger equation for quantum particles of closed quantum systems.

In Appendix we derive some simple estimates for traces used in our analysis.
 They are possibly known, but the author did not find an appropriate reference.

\subsection{Dynamics of pure states}
\label{secpureeq}

In this section we collect some (essentially known) facts about
the equations for pure states.
The pure states will be used as auxiliary tools in our analysis of mixed states.

The link between the two descriptions (linear and normalised) is summarised in the
following statement that can be checked by direct application
of classical Ito's lemma.

\begin{prop}
\label{linnorrmpure}
Let $H$ and $L$ be bounded.

(i) If $\chi(t)$ satisfies \eqref{eqqufiBlin},
then $\|\chi(t)\|^2$ satisfies the equations
 \begin{equation}
\label{chisquare}
d\|\chi(t)\|^2=2\sum_j(\chi(t),L_{Sj} \chi(t))dY_j(t),
\end{equation}
 \begin{equation}
\label{chisquare1}
d\frac{1}{\|\chi(t)\|^2}=-\frac{2}{\|\chi(t)\|^2}
\sum_j \frac{(\chi(t),L_{Sj} \chi(t))}{\|\chi(t)\|^2}
\left[dY_j(t)- \frac{(\chi(t),L_{Sj} \chi(t))}{\|\chi(t)\|^2} dt\right],
\end{equation}
 \begin{equation}
\label{chisquare2}
d\|\chi(t)\|=\frac{1}{\|\chi(t)\|}\sum_j(\chi(t),L_{Sj} \chi(t))dY_j(t)
-\frac{1}{2\|\chi\|^3} \sum_j(\chi(t),L_{Sj} \chi(t))^2 \, dt,
\end{equation}
and the normalised states $\phi(t)=\chi(t)/\|\chi(t)\|$ satisfy the
nonlinear equation \eqref{eqqufiBnonlin}, where
\begin{equation}
\label{eqdefinnov}
dB_j(t)=dY_j(t)-2(\phi(t), L_{Sj} \phi(t)) \, dt;
\end{equation}

(ii) Let $\phi(t)$ have unit norms for all $t$ and satisfy the
nonlinear equation \eqref{eqqufiBnonlin}. Define $\|\chi(t)\|^{-2}$ as the solution
(with the initial condition equal to $1$) to the equation
 \begin{equation}
\label{chisquare4}
d\frac{1}{\|\chi(t)\|^2}=-\frac{2}{\|\chi(t)\|^2}
\sum_j (\phi(t),L_{Sj} \phi(t)) dB_j(t),
\end{equation}
which is seen to be identical with \eqref{chisquare1}
when $B$ and $Y$ are linked via \eqref{eqdefinnov} and $\phi(t)=\chi(t)/\|\chi(t)\|$.
 Then the vectors
$\chi(t)=\phi(t) \|\chi(t)\|$ satisfy the linear equation \eqref{eqqufiBlin}.

(iii) The square norm $\|\chi(t)\|^2$ (respectively its inverse) of a solution
to \eqref{eqqufiBlin} is a martingale under
the probability law where $Y(t)$ (resp. $B(t)$) is a Brownian motion. If $B(t)$ is
a Brownian motin and $Y$ is given by \eqref{eqdefinnov}, then
 \begin{equation}
\label{chisquare5}
\E \|\chi(t)\|^2 \le \exp\{4 t \|L\|^2 \} \|\chi_0\|^2,
\end{equation}
where the expectation is with respect to $B(t)$.
\end{prop}

From now on we shall use reduced notation mentioned in introduction,
 tacitly assuming summation
over $n$ coordinates of $L$, $Y$ and $B$.

Notice that \eqref{eqqufiBnonlins} is meant to describe evolutions of unit trace.
 It is often convenient
to include it in a more general class of trace-preserving evolutions, the simplest version being
\[
d\phi(t)=-[i(H-\langle L_S \rangle_{\phi(t)} L_A)
+\frac12 (L-\langle L_S \rangle_{\phi(t)})^*(L-\langle L_S \rangle_{\phi(t)})]\phi(t) \, dt
\]
\begin{equation}
\label{eqqufiBnonlinsn}
+ (L-\langle L_S \rangle_{\phi(t)})\phi(t) \, dB(t),
\end{equation}
where we introduced the (rather standard) notation for the value
of an operator $A$ in a pure state $\phi$:
\[
\langle A \rangle_{\phi}=\frac{(\phi, A \phi)}{(\phi, \phi)}.
\]
Clearly for $\phi(t)$ of unit trace solutions to equations \eqref{eqqufiBnonlins}
and \eqref{eqqufiBnonlinsn} coincide, but \eqref{eqqufiBnonlinsn} is explicitly trace-preserving
for arbitrary $\phi(t)$, which is not the case for equation \eqref{eqqufiBnonlins}.

It is also insightful to write down equation for $\chi$ in terms of the innovation process $B$:
\begin{equation}
\label{eqqufiBlinsB}
d\chi(t) =-[iH\chi(t) +\frac12 L^*L \chi(t) ]\,dt
+L \chi(t) \, (dB(t)+\langle L+L^*\rangle_{\chi(t)} \, dt).
\end{equation}

In the most important case of a self-adjoint $L$ equations
\eqref{eqqufiBlins} and \eqref{eqqufiBnonlinsn} simplify to the equations
\begin{equation}
\label{eqqufiBlinss}
d\chi(t) =-[iH\chi(t) +\frac12 L^2 \chi(t) ]\,dt+L\chi(t) dY(t),
\end{equation}
and, respectively,
\begin{equation}
\label{eqqufiBnonlinss}
d\phi(t)=-[iH+\frac12 (L-(\phi(t), L \phi(t)))^2]\phi(t) \, dt
+ (L-(\phi(t), L \phi(t)))\phi(t) \, dB(t).
\end{equation}

Equations \eqref{eqqufiBlins} and \eqref{eqqufiBnonlinsn} may not make sense
if $H$ is unbounded. In the latter case some generalised version can be used. For instance,
one can look at mild forms of the Cauchy problem for these equations, which are
\begin{equation}
\label{eqqufiBlinsm}
\chi(t) =e^{-iH t}\chi_0+\int_0^t e^{-iH(t-s)} [-\frac12 L^*L \chi(s) \,dt+L \chi(s) \, dY(s)],
\end{equation}
and, respectively,
\[
\phi(t)=e^{-iH t}\phi_0+\int_0^t e^{-iH (t-s)} \bigl[ i\langle L_S \rangle_{\phi(s)} L_A \phi(s) \, ds
-\frac12 (L-\langle L_S \rangle_{\phi(s)})^*(L-\langle L_S \rangle_{\phi(s)})]\phi(s) \, ds
\]
\begin{equation}
\label{eqqufiBnonlinsm}
+ (L-\langle L_S \rangle_{\phi(s)})\phi(s) \, dB(s)\bigr].
\end{equation}

It is seen directly that in terms of vectors $\xi(t)=e^{iHt} \chi(t)$ and $\psi(t)=e^{iHt} \phi(t)$,
integral equations \eqref{eqqufiBlinsm} and \eqref{eqqufiBnonlinsm}, are equivalent
 to the Cauchy problems for the SDEs  in the "interaction form"
\begin{equation}
\label{eqqufiBlinsin}
d\xi(t) = -\frac12 (L^{Ht})^*L^{Ht} \xi(t) \,dt+L^{Ht} \xi(t) dY(t),
\end{equation}
and, respectively,
\[
d\psi(t)= i\langle L^{Ht}_S \rangle_{\psi(t)} L^{Ht}_A \psi(t) \, dt
-\frac12 (L^{Ht}-\langle L^{Ht}_S \rangle_{\psi(t)})^*(L^{Ht}-\langle L^{Ht}_S \rangle_{\psi(t)})]\psi(t) \, dt
\]
\begin{equation}
\label{eqqufiBnonlinsin}
+ (L^{Ht}-\langle L^{Ht}_S \rangle_{\psi(t)})\psi(t) \, dB(t),
\end{equation}
where $L^{Ht}$ are obtained from $L$ by "dressing":
\[
L^{Ht}=e^{iHt} L e^{-iHt}.
\]

On the other hand, a direct application of Ito's formula shows that, for bounded operators $H$,
SDEs \eqref{eqqufiBlinsin} and \eqref{eqqufiBnonlinsin} are equivalent to SDEs
\eqref{eqqufiBlins} and \eqref{eqqufiBnonlins}, respectively.
Of course,  for unbounded coefficients solutions
to mild equations may exist that fail to solve the corresponding SDEs.

Well-posedness for equation \eqref{eqqufiBlins}
with bounded $H,L$ and \eqref{eqqufiBlinsm} for bounded $L$ follows from the standard linear theory
and the observation that the r.h.s. are bounded linear operators. Let us note that much
more general equations of these
type with bounded coefficients are treated in detail in \cite{BarchHol}. Nonlinear equations are more subtle.
Possibly the well-posedness for general bounded $L$ was first obtained in \cite{KolQuantLLN} (in \cite{BarchHol}
the existence was proved and it was noted in \cite{MoraRebo} that the uniqueness, even of a weak solution,
was still open for bounded $H,L$ in infinite-dimensional spaces). The proof of the well-posedness
of \eqref{eqqufiBnonlinsm} (and thus \eqref{eqqufiBnonlins}) for bounded $L$ and self-adjoint $H$
in \cite{KolQuantLLN} was based on the extension to infinite-dimensional case of the arguments from \cite{BarchBook}
 given for finite-dimensional case. In fact, this well-posedness follows from the standard existence
 and uniqueness result for SDEs with Lipschitz coefficients and the following observation
 that we formulate as a separate lemma, because we shall use it also later in another context.

 \begin{lemma}
 \label{lemboundder}
 For any bounded operator $M$ in $\HC$, the mapping
 $\HC \to \HC$ given by formulas
 \[
 \psi \to f(\psi)= \langle M \rangle_{\psi} \psi
\]
is differentiable with the derivative mapping
\[
\phi  \to Df_{\psi}(\phi)=\frac{\pa f}{\pa \psi} \phi+ \frac{\pa f}{\pa \bar \psi} \bar \phi
\]
being a bounded linear operator with the norm not exceeding $5\|M\|$.
Consequently, the mapping $f$ is globally Lipschitz.
\end{lemma}

\begin{proof}
Writing down the derivatives in some orthonormal coordinates,
\[
\frac{\pa f_i}{\pa \psi_j}=\frac{(\psi, M\psi)}{(\psi, \psi)} \de^i_j
+\psi_i \frac{\overline{(M\psi)_j}}{(\psi, \psi)}- \psi_i \bar \psi_j\frac{(\psi, M\psi)}{(\psi, \psi)^2},
\]
we see that the first term gives the operator proportional to identity and
the second and third terms present Hilbert-Schmidt operators. Each of the
three terms is bounded by $\|M\|$. Similarly,
\[
\frac{\pa f_i}{\pa \bar \psi_j}
=\psi_i \frac{(M\psi)_j}{(\psi, \psi)}- \psi_i  \psi_j\frac{(\psi, M\psi)}{(\psi, \psi)^2},
\]
which is a Hilbert-Schmidt operator bounded by $2\|M\|$.
\end{proof}

A simple, but important observation is that SDEs \eqref{eqqufiBlinsin}
and \eqref{eqqufiBnonlinsin} have the same form as
SDEs \eqref{eqqufiBlins} and \eqref{eqqufiBnonlins} (the former have vanishing $H$ and
time dependent $L$), and therefore the calculations of Proposition \ref{linnorrmpure1}
extend automatically to equations \eqref{eqqufiBlinsin} and \eqref{eqqufiBnonlinsin}.
Moreover,
\[
\|\chi_t\|^2=\|\xi_t \|^2, \quad (\chi_t, L_S \chi_t)=(\xi_t, L_S^{Ht} \xi_t)
\]
leading to the following.

\begin{prop}
\label{linnorrmpure1}
Let $H$ be self-adjoint, but possibly unbounded and $L$ bounded.
All statements of Proposition \ref{linnorrmpure} remain valid literally.
\end{prop}

To complete our brief review of quantum filtering SDEs for pure states let us note
that there exist many important results on the solutions to some generalised
versions of equations \eqref{eqqufiBnonlinsm} with unbounded $H$ and $L$
under various nontrivial assumptions, see e.g.  \cite{Holevo91}, \cite{Fagnola},
\cite{MoraRebo}, \cite{Mora13}, some of them being inspired by the works on the
conservativity of quantum dynamic semigroups from \cite{ChebFagn} and \cite{ChebQuez}.
For other classes  of stochastic Schr\"odinger equation we can refer to \cite{BarbRock16}
and references therein.

\section{Stochastic Lindblad (or quantum master) equations: linear version}
\label{seclineq}

The first thing to decide for dealing with the equations on mixed states is the choice
of an appropriate Banach space of operators, where the corresponding SDEs will be analysed.

We shall consider these equations in the Hilbert space $\HC^2_s$ of self-adjoint Hilbert-Schmidt
operators in $\HC$, which is a closed subspace
in the space of all Hilbert-Schmidt operators $\HC^2$ with the scalar product ${\tr} (A^*B)$.

Since we are interested in trace-class operators, a more natural
space from physical point of view  would be the Banach space $\HC^1$ of self-adjoint
trace-class operators in $\HC$. However, the classes of Banach spaces,
for which a satisfactory extension of Ito stochastic calculus was developed,
namely the so-called UMD spaces, spaces of martingale type 2 and spaces with a smooth norm
(see review \cite{VanNeerven}) do not include $\HC^1$, and therefore we work in the larger space
$\HC^2$. In this space the key linear functional of taking trace is unbounded,
and we are led to work with SDEs with singular coefficients. This complication
is the price to pay for working in a convenient Hilbert setting of $\HC^2$.

As in the case of pure states, the terms of the equations for mixed states
are well defined for bounded $H$ and $L$.
For the case of unbounded $H$ and bounded $L$ one can naturally use the corresponding mild form
that writes down as
\[
\ga(t)=e^{-iHt}\ga_0e^{iHt}
+\int_0^t e^{-iH(t-s)} \LC_L\ga (s)  e^{iH(t-s)} \, ds
\]
\begin{equation}
\label{Lindstochmild}
 +\int_0^t e^{-iH(t-s)} (L\ga(s)+\ga(s) L^*) e^{iH(t-s)}\, dY(s)
\end{equation}
for equation \eqref{Lindstoch} and as
\[
\rho(t)=e^{-iHt}\rho_0 e^{iHt}
+\int_0^t e^{-iH(t-s)} \LC_L\rho (s)  e^{iH(t-s)} \, ds
\]
\begin{equation}
\label{Lindstochnormmild}
+\int_0^t e^{-iH(t-s)} \bigl[
L\rho(s)+\rho(s) L^*-\rho(s)\, {\tr} \, (L\rho(s)+\rho(s) L^*) \bigr] e^{iH(t-s)} dB(s)
\end{equation}
for equation \eqref{Lindstochnorm1}.

Also insightful are the versions of these equations in "interaction form", which are
\begin{equation}
\label{Lindstochin}
d\nu(t)=\LC_{L^{Ht}} \nu(t) \, dt +(L^{Ht}\nu(t)+\nu(t) (L^{Ht})^*) dY(t),
\end{equation}
for equation \eqref{Lindstochmild}, written in terms of $\nu(t)=e^{-iHt}\ga(t) e^{iHt}$, and
\begin{equation}
\label{Lindstochnormin}
d\mu(t)=\LC_{L^{Ht}} \mu (t)\, dt
+[L^{Ht}\mu(t)+\mu(t) (L^{Ht})^*-\mu(t)\, {\tr} \, (L^{Ht}\mu(t)+\mu(t) (L^{Ht})^*) ] dB(t)
\end{equation}
for equation \eqref{Lindstochnormmild}, written in terms of
$\mu(t)=e^{-iHt}\rho(t) e^{iHt}$. Here
\[
L^{Ht}=e^{iHt} L e^{-iHt},
\]
as above.
The equivalence of \eqref{Lindstochin} (resp. \eqref{Lindstochnormin}) and \eqref{Lindstochmild}
(resp. \eqref{Lindstochnormmild}) is straightforward, but
equations  \eqref{Lindstochin} and \eqref{Lindstochnormin} can be looked at as particular cases
of  \eqref{Lindstoch} and \eqref{Lindstochnorm1}, though with time-dependent coefficients.

\begin{theorem}
\label{LindStochLin1}
Assume, as usual, that $L$ is bounded, $H$ is self-adjoint, and $Y(t)$ is a simple Ito's process.
 Then the following holds.

(i) Equation \eqref{Lindstoch} in case of bounded $H$ and equations
\eqref{Lindstochmild} or \eqref{Lindstochin} in general case are well-posed in $\HC^2_s$, that is,
they have a unique global solution for any $\ga_0\in \HC^2_s$. If $Y(t)$ is a Brownian motion,
then these solutions have the growth estimates
 \begin{equation}
\label{Lindstochquadex1}
\E [{\tr}\, \ga^2(t)] \le {\tr}\, \ga_0^2 \exp\{ 4 t\|L\|^2\}.
\end{equation}

(ii) The solution $\ga (t)$ to  \eqref{Lindstochmild} or \eqref{Lindstoch}
is positive-definite for all $t$ whenever $\ga_0$ so is.

(iii) If the initial condition $\ga_0$ is of trace-class,
then so is the solution $\ga(t)$, with the trace given by the formula
\begin{equation}
\label{Lindstochmart}
{\tr} \,\ga(t)={\tr} \, \ga_0+\int_0^t {\tr} (L\ga(s)+\ga(s) L^*) dY(s).
\end{equation}
If $Y(t)$ is a Brownian motion, then ${\tr} \, \ga(t)$ is a
square integrable martingale such that
\begin{equation}
\label{Lindstochmart1}
\E ({\tr} \,\ga(t))^2\le [({\tr} \, \ga_0^+)^2 +({\tr} \, \ga_0^-)^2]\exp\{ 4t \|L\|^2\},
\end{equation}
where $\ga_0^{\pm}$ denote positive and negative parts of $\ga_0$, and
\begin{equation}
\label{Lindstochmart2}
\E ({\tr} \,|\ga(t)|)\le {\tr} \, |\ga_0|.
\end{equation}

\end{theorem}

\begin{remark}
All estimates in the theorem can be extended to arbitrary simple
processes $Y(t)$ by taking into account the bound for the coefficient
at $dt$ of its differential.
\end{remark}

\begin{remark}
The proof of positivity of the solutions was usually considered as a difficult task.
Three different approaches for such proof in finite-dimensional case were suggested e.g. in
\cite{BarchBook} (using a link with pure state equations), \cite{Pellegrini10} (using discrete
Markov  chain approximation), \cite{KolQuantFrac} (turning to Stratonovich SDE or using the theory
of attainable boundary points). We give here a simple proof that works also in infinite dimensional case,
It is close in spirit to the proof of \cite{BarchBook}.
\end{remark}

\begin{proof}
(i) The existence of a unique solution is straightforward, as the coefficients at $dt$ and $dY$ are
bounded linear operator of $\ga$ in $\HC^2_s$.

We shall work with equation \eqref{Lindstoch}, extension to  \eqref{Lindstochin} being automatic.

We derive from
\eqref{Lindstoch} by Ito's formula that
\[
d \, {\tr}\, \ga^2={\tr} \,(\ga L^*\ga L+\ga L\ga L^*+ L \ga L \ga + \ga L^* \ga L^*) dt
\]
\begin{equation}
\label{Lindstochquadtr}
+{\tr} \, [L\ga^2+\ga L^*\ga+\ga L\ga+\ga^2 L^*]  dY_t,
\end{equation}
and thus
\[
\E \, {\tr}\, \ga^2(t)
={\tr}\, \ga_0^2 +\E \, {\tr} \,
\int_0^t (\ga(s) L^*\ga(s) L+\ga(s) L\ga(s) L^*+ L \ga(s) L \ga(s) + \ga(s) L^* \ga(s) L^*) ds
\]
so that, by \eqref{eqmyineqtr1},
\[
\E \, {\tr}\, \ga^2(t) \le {\tr}\, \ga_0^2+ 4 \|L\|^2 \int_0^t  \E \, {\tr} \,(\ga^2(s)) \, ds.
\]
and \eqref{Lindstochquadex1} follows by By Gronwall's lemma.

 (ii) Since $\ga_0$ is a positive operator from $\HC^2_s$, it follows that
  there exists a orthonormal basis $\{e_k\}$ in $\HC$ such that $\ga_0$
 can be presented as a convergent (in $\HC^2_s$) series
 \[
 \ga_0=\sum_{k=1}^{\infty} p_k e_k\otimes \bar e_k
 \]
 with non-increasing non-negative sequence $\{p_k\}$ from $l^2$. Hence $\ga_0=\lim \ga_{0n}$
 with finite-dimensional operators
 \[
 \ga_{0n}=\sum_{k=1}^n p_k e_k\otimes \bar e_k.
 \]
By linearity (and uniqueness of solutions), the solution $\ga_n(t)$ with the initial condition $\ga_{0n}$ is the finite
convex combination of the solutions with the initial conditions $e_k\otimes \bar e_k$, the latter
being given by $e_k(t)\otimes \bar e_k(t)$ with $e_k(t)$ solving the linear filtering equation for
pure states \eqref{eqqufiBlins}, and thus being positive definite. Therefore, $\ga_n(t)$ are positive-definite.
On the other hand, by \eqref{Lindstochquadex1}, $\ga_n(t)-\ga(t)$ tend to zero, as $n\to \infty$,
as the solutions with initial condition $\ga_0-\ga_{0t}$ tending to zero. Hence all $\ga (t)$ are
also positive-definite.

(iii) First assume that $\ga_0$ is positive. We then use the approximations
$\ga_n(t)$ as defined in (ii) above.
Since all $e_k(t)\otimes \bar e_k(t)$ are positive-definite operators,
the sequence $\ga_n(t)$ is monotonically increasing in $n$.

Since
\begin{equation}
\label{tracenormsq}
{\tr} \, e_k(t)\otimes \bar e_k(t)=\|e_k(t)\|^2,
\end{equation}
it follows that, for $n>m$,
\[
{\tr}\, |\ga_n(t)-\ga_m(t)|={\tr}\, (\ga_n(t)-\ga_m(t))
=\sum_{k=m+1}^n p_k \|e_k(t)\|^2,
\]
and thus the sequence $\ga_n(t)$ converges not only in $\HC^2$, but also in
the space of trace-class operators $\HC^1$. Hence, $\ga(t)$ is of trace class and
${\tr }\, \ga(t)=\lim {\tr}\, \ga_n(t)$.

From \eqref{chisquare} and \eqref{tracenormsq} it follows that
\[
{\tr} \,\ga_n(t)={\tr} \, \ga_{0n}+\int_0^t {\tr} (L\ga_n(s)+\ga_n(s) L^*) dY(s).
\]
Passing to the limit in this equation
we obtain \eqref{Lindstochmart}.

If $Y(t)$ is a Brownian motion, then from \eqref{Lindstochmart}, Ito's formula and the estimate
$|{\tr} (\ga L)|\le {\tr}\, \ga \,  \|L\|$, it follows that
 \[
 \E ({\tr}\, \ga(t))^2
 \le  ({\tr} \, \ga(0))^2+4 \|L\|^2 \int_0^t \E ({\tr} \, \ga(s))^2\,  ds
\]
implying \eqref{Lindstochmart1} by Gronwall's lemma.

Finally, if $\ga_0$ is not positive, we decompose $\ga_0$
as the difference of its positive and negative parts:
$\ga_0=\ga_0^+-\ga_0^-$. Applying  \eqref{Lindstochmart}
 to the corresponding solutions $\ga^{\pm}(t)$ , we get  \eqref{Lindstochmart} for $\ga(t)$
by linearity. Similarly estimates \eqref{Lindstochmart1} are obtained. It remains estimate \eqref{Lindstochmart2}.
Firstly it clearly holds with the sign of equality for positive $\ga_0$. For general Hermitian $\ga_0$,
we can write
\[
{\tr} \, |\ga(t)| \le {\tr} \, \ga^+(t) +{\tr} \, \ga^-(t).
\]
The sign of inequality is due to the fact that, though the solutions $\ga^{\pm}(t)$ are positive operators,
 they are not necessarily positive and negative parts of the operator $\ga(t)$).
Applying  \eqref{Lindstochmart2} to $\ga^{\pm}(t)$ we get \eqref{Lindstochmart2} to $\ga(t)$.
 \end{proof}

\begin{remark}
Similarly to the arguments of (i) one can establish the continuity estimates
\begin{equation}
\label{Lindstochcontex}
\E [{\tr}\, (e^{-iHt}\ga(t)e^{iHt}-e^{-iHs}\ga(s)e^{iHs})^2]
\le 8 (t-s) (t\|L\|^4 +\|L\|^2) {\tr}\, \ga_0^2 \exp\{ 4t \|L\|^2\}.
\end{equation}
for $0\le s\le t$, and in case of bounded $H$,
\begin{equation}
\label{Lindstochcontexb}
\E [{\tr}\, (\ga(t)-\ga(s))^2] \le 12 (t-s) [t (\|H\|^2+\|L\|^4) +\|L\|^2] {\tr}\, \ga_0^2 \exp\{ 4t \|L\|^2\}.
\end{equation}
\end{remark}

Important part of well-posedness of an equation is the continuous dependence of its solution on
 initial conditions and parameters. By linearity, continuous dependence on initial condition
 follows from \eqref{Lindstochquadex1}. Next result establishes the continuous dependence on
 a bounded part of the Hamiltonian.

\begin{theorem}
\label{LindStochLin2} Under assumptions of Theorem \ref{LindStochLin1}
assume $Y(t)$ is a Brownian motion and
consider two equations of type \eqref{Lindstoch}:
\begin{equation}
\label{Lindstochtwo}
d\ga(t)=-i[H,\ga] \, dt -i[H_j,\ga] \, dt+\LC_L \ga(t) \, dt +(L\ga(t)+\ga(t) L^*) dY(t),
\end{equation}
$j=1,2$, where $H_1$ and $H_2$ are two bounded self-adjoint operators in $\HC$.
Then for their solutions $\ga_j(t)$, $j=1,2$, with one and the same positive initial
condition $\ga_0$ of trace-class
we have the estimates for the deviations in the norms of $\HC^1$ and $\HC^2$:
\begin{equation}
\label{eqLindstochLin21}
\E \, {\tr}\, |\ga_1(t)-\ga_2(t)|\le 2 t \|H_2-H_1\| \, {\tr} \, \ga_0,
\end{equation}
\begin{equation}
\label{eqLindstochLin22}
\sqrt{\E \, {\tr}\, (\ga_1(t)-\ga_2(t))^2}
\le 2 t \|H_2-H_1\|\sqrt{ {\tr} \, \ga_0^2} \exp\{2t \|L\|^2\}.
\end{equation}
\end{theorem}

\begin{proof}
By turning to the interaction representation we can reduce the story to the case $H=0$.
Then by subtracting two equations we get the following:
\[
d(\ga_1(t)-\ga_2(t))=-i[H_1,\ga_1(t)-\ga_2(t)] \, dt
+\LC_L (\ga_1(t)-\ga_2(t)) \, dt +(L(\ga_1(t)-\ga_2(t))
\]
\begin{equation}
\label{Lindstochtwo1}
+(\ga_1(t)-\ga_2(t)) L^*) dY(t)
+i[H_2-H_1, \ga_2(t)] \, dt.
\end{equation}
Denoting by $\Phi_t$ the operator giving solution to the Cauchy problems for equation \eqref{Lindstochtwo}
with $H=0$ and $j=1$, we can express solution to \eqref{Lindstochtwo1} with the vanishing
initial condition in the following standard Du Hamel form:
\begin{equation}
\label{Lindstochtwo2}
\ga_1(t)-\ga_2(t)=i\int_0^t \Phi_{t-s} [H_2-H_1, \ga_2(s)] \, ds.
\end{equation}
Hence
\[
\E \, {\tr}\, |\ga_1(t)-\ga_2(t)|\le \int_0^t \E \, {\tr} \, |\Phi_{t-s} [H_2-H_1, \ga_2(s)]| \, ds.
\]
By the chain rule,  one can insert the conditional expectation with
respect to $\FC_s$ inside the expectation on the r.h.s. of the inequality
and then apply \eqref{Lindstochmart2} leading to the estimate
\[
\E \, {\tr}\, |\ga_1(t)-\ga_2(t)|\le \int_0^t \E \, {\tr} \, | [H_2-H_1, \ga_2(s)]| \, ds
\]
\[
\le 2\|H_2-H_1\| \int_0^t \E \, {\tr} \, \ga_2(s) \, ds
\le 2 t \|H_2-H_1\|  \, {\tr} \, \ga_0,
\]
implying \eqref{eqLindstochLin21}.

Similarly, using \eqref{Lindstochtwo2}, insertion of conditional expectation
and estimates  \eqref{Lindstochquadex1} and \eqref{eqmyineqtr2}, we write
\[
\sqrt{\E \, {\tr}\, (\ga_1(t)-\ga_2(t))^2}
\le \int_0^t \sqrt{ \E \, {\tr} \, (\Phi_{t-s} [H_2-H_1, \ga_2(s)])^2} \, ds
\]
\[
\le \int_0^t \sqrt{\E \, {\tr} \, ([H_2-H_1, \ga_2(s)])^2} \exp\{2(t-s)\|L\|^2\} \, ds
\]
\[
\le 2\|H_2-H_1\| \int_0^t \sqrt{\E \, {\tr} \, (\ga_2(s))^2} \exp\{2(t-s)\|L\|^2\} \, ds
\]
\[
\le 2 \|H_2-H_1\| \sqrt{{\tr} \, \ga_0^2} \int_0^t \exp\{2t\|L\|^2\} \, ds
\]
implying \eqref{eqLindstochLin22}.
\end{proof}

\begin{remark} An alternative proof of \eqref{eqLindstochLin22} can be given by writing
down the equation for ${\tr} \, (\ga_1(t)-\ga_2(t))^2$ and then applying Gronwall's lemma.
\end{remark}

\section{Stochastic Lindblad (or quantum master) equations: normalised version}
\label{secnormeq}

We shall now establish the link between linear and nonlinear master equation in analogy
with the case of pure states dealt above.

From \eqref{Lindstochmart} we can derive by Ito's formula that
\[
d\frac{1}{{\tr} \, \ga(t)}=-\frac{1}{({\tr} \, \ga(t))^2} \, {\tr} \, (L\ga(t) +\ga(t) L^*) dY_t
+\frac{1}{({\tr} \, \ga(t))^3} [{\tr} \, (L\ga(t)+\ga(t) L^*)]^2 \, dt
\]
Hence by Ito's product rule we check that the normalised density operator
$\rho(t)=\ga(t)/{\tr} \, \ga(t)$ satisfies the equation
 \[
d\rho=-i[H, \rho(t)] \, dt +\LC_L \rho(t)\, dt
\]
 \begin{equation}
\label{Lindstochnorm}
+(L\rho(t)+\rho(t) L^*-\rho(t)\, {\tr} \, (L\rho(t)+\rho(t) L^*) )
[dY_t-{\tr} \, (L\rho(t)+\rho(t) L^*) dt].
\end{equation}

Therefore, in terms of the {\it innovation process}
 \begin{equation}
\label{outputinnovation}
B(t)=Y(t)-\int_0^t {\tr} \, (L\rho(s)+\rho(s) L^*) \, ds
\end{equation}
 the equation for the inverse trace rewrites as
  \begin{equation}
\label{eqtrinnov}
d\frac{1}{{\tr} \, \ga(t)}=-\frac{1}{{\tr} \, \ga(t)} \, {\tr} \, (L\rho(t) +\rho(t) L^*) dB(t)
\end{equation}
and the equation for the normalised density operator \eqref{Lindstochnorm} rewrites in the
standard form \eqref{Lindstochnorm1} of the nonlinear filtering equation.

Let us summarise the corresponding calculations performed above in the following statement that is
a mixed state analog of Propositions \ref{linnorrmpure} and \ref{linnorrmpure1}.

\begin{prop}
\label{linnonlin}
Suppose, as usual, that $L$ is bounded and $H$ is self-adjoint. (i) Let $\ga(t)$ satisfy the linear equation
\eqref{Lindstochin} (or \eqref{Lindstoch} in case of bounded $H$) in terms of a simple Ito process $Y(t)$,
and have positive initial condition of trace-class $\ga_0$. Then ${\tr }\, \ga(t)$
 satisfies \eqref{Lindstochmart}. Moreover, the processes $\rho(t)=\ga(t)/{\tr} \, \ga(t)$
 and $B(t)$ defined via \eqref{outputinnovation} satisfy the nonlinear equation \eqref{Lindstochnormin}
 (and \eqref{Lindstochnorm1} in case of bounded $H$).
 (ii) Let $\rho(t)$ satisfy the nonlinear
equation \eqref{Lindstochnormin} (or \eqref{Lindstochnorm1} in case of bounded $H$)
in terms of a simple Ito process $B(t)$, have positive initial condition
$\rho_0$ and have unit trace for all $t$. Then, if the process ${\tr }\, \ga(t)$ is defined
as the solution of equation \eqref{eqtrinnov} (with any positive initial condition), the process
$\ga(t)$ is defined as $\ga(t)= \rho(t) {\tr }\, \ga(t)$ and the process
$Y(t)$ is defined via \eqref{outputinnovation}, then these processes satisfy
the linear equation \eqref{Lindstoch}.
\end{prop}

The main complication with nonlinear equation \eqref{Lindstochnorm1}
is due to the fact that the function ${\tr} \, (L\rho+\rho L^*)$ appearing in
its r.h.s. is not Lipschitz continuous as a function of $\rho$ in our main working
space $\HC^2$ unless $L$ itself belongs to this space. In this particular case one
can build the theory of nonlinear equation \eqref{Lindstochnorm1} in analogy with the
linear equation. For general bounded $L$ we shall work via the correspondence of
Proposition \ref{linnonlin} starting with the following result.

\begin{prop}
\label{linnonlin1}
Let $L$ be bounded, $H$ self-adjoint, and $B(t)$ be a simple Ito process.
For any adapted continuous process $\rho(t)$ in $\HC^2_s$
 with all $\rho(t)$ positive and of unit trace, let us define the process $T(t)={\tr }\, \ga(t)$ as
the solution of equation \eqref{eqtrinnov} with the initial condition equal to $1$, the process
$\ga(t)$ as $\ga(t)=T(t)  \rho(t)$ and the process
$Y(t)$ via \eqref{outputinnovation}. Let $\tilde \ga(t)$ be the solution
of the linear equation \eqref{Lindstoch} with the initial condition $\ga_0=\rho_0$. Set
$\tilde \rho(t)=\tilde \ga(t) / {\tr}\, \tilde \ga(t)$ and define $\tilde B$
via equation \eqref{outputinnovation} with $\tilde \rho$ instead of $\rho$.
 Then the process $\rho(t)$
is a fixed point of the mapping $\rho(.) \to \tilde \rho(.)$ if and only if $\rho(.)$ solves
the nonlinear equation \eqref{Lindstochnorm1}.
\end{prop}

\begin{remark} Similarly one can state that the pair of processes $(\rho(t), B(t))$
is a fixed point of
the mapping $(\rho(.), B(.)) \to (\tilde \rho(.), \tilde B(.))$ if and only if
the pair $(\rho(t), B(t))$ solves \eqref{Lindstoch}.
\end{remark}

\begin{proof} (i) Let $\tilde \rho(.)=\rho(.)$. Then also $B(t)=\tilde B(t)$
and $(\rho(t), B(t))$ solves \eqref{Lindstochnorm1}.
(ii) Let $\rho(t)$ solves \eqref{Lindstochnorm1}. Then, by Proposition \ref{linnonlin},
$\ga(t)$ solves \eqref{Lindstoch}. By Theorem  \ref{LindStochLin1}, $\tilde \ga(t)=\ga(t)$.
\end{proof}

The link with the linear equation allows one to derive the well-posedness
of nonlinear equation \eqref{Lindstochnorm1} in the sense of weak solutions
as a direct consequence of Theorem  \ref{LindStochLin1}. Recall that a strong
solution of an SDE like \eqref{Lindstochnorm1} is a process $\rho(t)$ that can
 be expressed as a measurable function of a given Brownian motion $B(t)$. By a
weak solution one means a pair of processes $(\rho(t),B(t))$ (where $B(t)$ is a Brownian
motion) defined on a certain stochastic basis, adapted to its filtration and satisfying
\eqref{Lindstochnorm1}. Recall that one says that weak solution is  unique in law
if for any two solutions $(\rho^1,B^1)$ and $(\rho^2,B^2)$ (possibly defined on
different probability spaces) the processes $\rho^1$ and $\rho^2$ have the same distribution.

\begin{theorem}
\label{LindStochnonLin}
Let $H$ be self-adjoint, $L$ bounded and $\rho_0$ a positive-definite operator of unit trace.
Then there exists a unique in law weak solution of equation \eqref{Lindstochnorm1} in $\HC^2$
with $B(t)$ a Brownian motion, the initial data $\rho_0$ and such that all $\rho(t)$
are positive-definite operators of unit trace.
\end{theorem}

\begin{proof} By Theorem  \ref{LindStochLin1} one can construct a strong solution to linear equation
\eqref{Lindstoch} with $Y(t)$ a Brownian motion. By Proposition \ref{linnonlin1},
$\rho(t)=\ga(t)/{\tr} \, \ga(t)$ satisfies equation \eqref{Lindstochnorm1}. By Girsanov's theorem,
$B(t)$ will be a Brownian motion under an appropriate change of equivalent measure.
This proves the existence of a solution. By Proposition \ref{linnonlin1},
any solution can be obtained in this way. Since the distribution of $\ga(t)$ is fixed
by Theorem  \ref{LindStochLin1}, the distribution of $\rho(t)$ is also uniquely defined.
\end{proof}

Let us turn to the strong solutions of \eqref{Lindstochnorm1}.

With Proposition \ref{linnonlin1} in mind, one could naturally try to prove the existence of
solutions to \eqref{Lindstochnorm1} by proving the existence
of a fixed point of the mapping $\rho\to \tilde \rho$. However, a direct attempt to prove
 the contraction property of this mapping would meet the same difficulty mentioned above, namely the fact
 that the function ${\tr} \, (L\rho+\rho L^*)$ is not Lipschitz continuous.
 However, Proposition \ref{linnonlin1} can be used to reveal the structure of \eqref{Lindstochnorm1}
 and then eventually reformulate it in terms of a system of equations with Lipschitz continuous
 coefficients. This program is carried out in the next theorem. Its importance lies not only in
 its stated result, but in the main tool in the proof, which is the possibility to rewrite stochastic master
 equation for mixed states in the equivalent form that coincides with the corresponding equation for pure
  states in an appropriately chosen Hilbert space.

 \begin{theorem}
\label{LindStochnonLin1}
Let $H$ be self-adjoint, $L$ bounded, $\rho_0$ a positive-definite operator of unit trace,
and $B(t)$ a simple Ito process. Then there exists a unique strong solution of equation
\eqref{Lindstochnorm1} in $\HC^2$, with the initial data $\rho_0$
and such that all $\rho(t)$ are positive-definite trace class operators of unit trace.
\end{theorem}

\begin{proof}
According to Proposition \ref{linnonlin1}, $\rho(t)$ solves \eqref{Lindstochnorm1}
 if and only if $\ga(t)=T(t) \rho(t)$ solves the equation
  \begin{equation}
\label{Lindstochnormlinmix}
d\ga(t)=-i[H, \ga(t)] dt +\LC_L \ga(t) dt +(L\ga(t)+\ga(t) L^*)
[dB(t)+ \pi(t) \, dt],
\end{equation}
with
\[
\pi(t)=T^{-1}(t)\, {\tr} \, (L\ga(t)+\ga(t) L^*).
\]

As in the proof of Theorem \ref{LindStochLin1} above, we can
expand $\rho_0=\ga_0$ in a series
\[
 \ga_0=\sum_{k=1}^{\infty} p_k e_k\otimes \bar e_k
 \]
 with a non-negative sequence $\{p_k\}$ summing up to one
 and an orthonormal basis $\{e_k\}$.
Hence we can represent $\ga(t)$ as the convergence series of pure states
 \[
 \ga(t)=\sum_{k=1}^{\infty} p_k e_k(t)\otimes \bar e_k(t),
 \]
 with $e_k(t)$ solving the linear filtering equation for
pure states \eqref{eqqufiBlins}:
\begin{equation}
\label{eqinfdim}
de_k(t)=(-iH e_k(t)-\frac12 L^*L e_k(t))\,dt +Le_k(t) [dB(t)+\pi(t)\, dt].
\end{equation}
Here
\[
\pi(t)=\frac{\sum_{k=1}^{\infty} p_k (e_k(t), (L+L^*) e_k(t))}{\sum_{k=1}^{\infty} p_k \|e_k(t)\|^2}.
\]

It is insightful to consider the infinite-dimensional system of SDEs \eqref{eqinfdim} as a single SDE
with values in the  Hilbert space $l^2_{\HC}(\{p_k\})$ consisting of infinite sequences
$\e=(e_1, e_2, \cdots )$ of vectors from $\HC$ and equipped
with the norm
\[
\|\e\|^2=\sum_{k=1}^{\infty} p_k (e_k, e_k).
\]
Bounded operators in $\HC$ extend naturally (acting identically on each coordinate)
to bounded operators in $l^2_{\HC}(\{p_k\})$ with the preservation of norm.
In this notation system \eqref{eqinfdim} writes down as the SDE
\begin{equation}
\label{eqinfdim1}
d \e(t)=(-iH \e(t)-\frac12 L^*L \e(t))\, dt +L\e(t) \left[dB(t)+\frac{(\e, (L+L^*) \e)}{(\e, \e)} \, dt\right].
\end{equation}

This equation is the same as \eqref{eqqufiBlinsB} (though written in an enhanced Hilbert space). Hence
the coefficients of this equation are globally Lipschitz due to Lemma \ref{lemboundder}.
Therefore, it has the unique solution. And consequently, \eqref{Lindstochnormlinmix} has a unique solution $\ga(t)$,
and hence  $\rho(t)=\ga(t)/T(t)$ is the unique solution to the Cauchy problem for \eqref{Lindstochnorm1}.
\end{proof}

The full well-posedness of a problem includes also a statement on a continuous dependence of
the solution on initial data and parameters of the problem. We shall prove here continuous dependence
on the Hamiltonian, which would be crucial for the next Section.

 \begin{theorem}
\label{LindStochnonLin2}
Under the assumption of Theorem \ref{LindStochnonLin1}
let us consider the Cauchy problem for equations
\begin{equation}
\label{Lindstochnorm11}
d\rho(t)=-i[H+H_j,\rho(t)]\, dt+\LC_L \rho (t)\, dt
+[L\rho(t)+\rho(t) L^*-\rho(t)\, {\tr} \, (L\rho(t)+\rho(t) L^*) ] dB(t),
\end{equation}
$j=1,2$, where $H_1, H_2$ are bounded self-adjoint operators in $\HC$.
Then for the solutions $\rho_j(t)$, $j=1,2$, of these equations with the same
positive initial data $\rho_0$ of unit trace one has the following estimate:
\begin{equation}
\label{Lindstochnorm12}
\E \|\rho_1(t)-\rho_2(t)\|_{H^{1,2}_s} \le \sqrt t C(t) \|H_1-H_2\|,
\end{equation}
where $C(t)$ is an increasing continuous function depending on $\|L\|$,
and where $\HC^{1,2}_s$ means of course either $\HC^1_s$ or $\HC^2_s$.
\end{theorem}

\begin{proof} By changing to the "interaction picture", that is, to equations
on the variable $e^{-itH}\rho(t) e^{itH}$ we can reduce the discussion to the case
of vanishing $H$. Time-dependence of $H_j$ and $L$ arising from this change does
not affect the argument. Therefore, without loss of generality we can set $H=0$.


Making the transformation to the equations in $l^2_{\HC}(\{p_k\})$,
as in the proof of the previous theorem,
 we can rewrite equations \eqref{Lindstochnorm11} as the equations
\begin{equation}
\label{eqinfdim11}
d \e^j(t)=(-iH_j \e^j(t)-\frac12 L^*L \e^j(t))\, dt +L\e^j(t)
\left[dB(t)+\frac{(\e^j, (L+L^*) \e^j)}{(\e^j, \e^j)} \, dt\right],
\end{equation}
with $j=1,2$.

As shown in the previous theorem, these equations are SDEs in a Hilbert space
with globally Lipschitz coefficients. Moreover, these two equations differ by
bounded linear terms. Hence it is a standard procedure
(see e.g. Proposition 7.1 in \cite{KolQuantLLN}) to derive an
estimate for $\E \|(\e^1-\e^2)(t)\|^2$ in terms of $\|H_1-H_2\|$.
Let us sketch the derivation for completeness. From \eqref{eqinfdim11} it follows that
\[
d(\e^1-\e^2)(t)=-iH_1 (\e^1-\e^2)(t)\, dt -\frac12 L^*L(\e^1-\e^2)(t)\, dt +L(\e^1-\e^2)(t) \, dB(t)
\]
\[
+ [G(e^1(t))-G(e^2(t))] \, dt -i(H_1-H_2)\e^2(t) \, dt,
\]
where we denoted
\[
G(\e)=\frac{(\e, (L+L^*) \e)}{(\e, \e)} L\e,
\]
which is known from Lemma \ref{lemboundder} to be Lipschitz continuous mapping with
the Lipschitz constant $G_L=10 \|L\|^2$. Consequently,
 \[
d\|(\e^1-\e^2)(t)\|^2= ((L^*+L)(\e^1-\e^2)(t), (\e^1-\e^2)(t)) \, dB(t)
\]
\[
+ 2\, Re \, ((\e^1-\e^2)(t), G(\e^1(t))-G(\e^2(t)))\, dt
-2 \, Re \, (i(H_1-H_2)\e^2(t), (\e^1-\e^2)(t)) \, dt.
\]
Consequently,
\[
d \, \E \|(\e^1-\e^2)(t)\|^2
\le 2G_L  \E \|(\e^1-\e^2)(t)\|^2 \, dt
+ 2\|H_1-H_2\| \E (\|e^2(t)\| \, \|(\e^1-\e^2)(t)\|).
\]
Using Cauchy inequality and estimate \eqref{chisquare5}, yields the following estimate:
\[
d \, \E \|(\e^1-\e^2)(t)\|^2
\le 2G_L  \E \|(\e^1-\e^2)(t)\|^2 \, dt +  4\E \|(\e^1-\e^2)(t)\|^2 \, dt
+ 4\|H_1-H_2\|^2 \E \|e^2(t)\|^2 \, dt
\]
 \[
 \le (2G_L+4)  \E \|(\e^1-\e^2)(t)\|^2 \, dt +4\|H_1-H_2\|^2 e^{4 t \|L\|^2} \, dt.
\]
Consequently,
\[
\E \|(\e^1-\e^2)(t)\|^2 \le 4 \|H_1-H_2\|^2 \int_0^t \exp\{(20 \|L\|^2+4)(t-s)+4s\|L\|^2\} \, ds
\]
\begin{equation}
\label{eqinfdim12}
\le  4 t \|H_1-H_2\|^2  e^{(20 \|L\|^2+4) t}.
\end{equation}
Similarly one gets the estimate
\begin{equation}
\label{eqinfdim13}
\E \|(\e^1-\e^2)(t)\|^4 \le t C(t) \|H_1-H_2\|^2
\end{equation}
with an increasing continuous function $C(t)$ depending on $\|L\|$. Let us stress that the norms
of vectors $\e^j$ are of course the norms in the space  $l^2_{\HC}(\{p_k\})$.

Next, recall that
\[
\rho_j(t) =\frac{\sum_{k=1}^{\infty} p_k (e_k^j(t) \otimes \bar e_k^j(t))}{\sum_{k=1}^{\infty} p_k \|e_k^j(t)\|^2}.
\]
We need an estimate for $\|\rho_1-\rho_2\|$ in terms of $\|\e^1-\e^2\|$. In the calculations below we shall not
write argument $t$ explicitly.
Possibility to have identical estimates for the norms of $\HC^1_s$ and $\HC^2_s$ comes from the observation that
\[
\|x\otimes y\|_{\HC^1_s}={\tr}\, |x\otimes y |=\|x\otimes y\|_{\HC^2_s}=\sqrt{ {\tr} \, (x\otimes y)^2}=|(x,y)|.
\]
for any two vectors $x,y$ in a Hilbert space.

Consequently it follows that
\[
\|\sum_{k=1}^{\infty} p_k (e_k^1 \otimes  \bar e_k^1)
-\sum_{k=1}^{\infty} p_k (e_k^2 \otimes  \bar e_k^2)\|_{H_s^{1,2}}
\le
\|\e^1-\e^2\| (\|\e^1\|+\|\e^2\|),
\]
and therefore
\[
\|\rho_1-\rho_2\|_{\HC^{1,2}_s}
\le \frac{\|\e^1-\e^2\| (\|\e^1\|+\|\e^2\|)}{\|\e^1\|^2}
+\|\e_2\|^2 \left| \frac{1}{\|\e^2\|^2}-\frac{1}{\|\e^1\|^2}\right|.
\]
Since
\[
 \left| \frac{1}{\|\e^2\|^2}-\frac{1}{\|\e^1\|^2}\right|
 \le \frac{\|\e^2-\e^1\| (\|e^2\|+\|e^1\|)}{\|e^2\|^2\|e^1\|^2},
 \]
 it follows that
 \[
\|\rho_1-\rho_2\|_{\HC^{1,2}_s}
\le 2 \|\e^1-\e^2\| \frac{(\|\e^1\|+\|\e^2\|)}{\|\e^1\|^2}.
\]
Estimating $\|\e^2\|\le \|\e^1\|+\|\e^1-\e^2\|$ we derive further that
\[
\|\rho_1-\rho_2\|_{\HC^{1,2}_s}\le 4  \frac{\|\e^1-\e^2\|}{\|\e^1\|}
+ 2  \frac{\|\e^1-\e^2\|^2}{\|\e^1\|^2}.
\]
Consequently,
\[
\E \|\rho_1-\rho_2\|_{\HC^{1,2}_s}
\le 4  \sqrt{ \E \|\e^1-\e^2\|^2} \, \sqrt {\E \frac{1}{\|\e^1\|^2}}
+ 2  \sqrt{ \E\|\e^1-\e^2\|^4} \, \sqrt{ \E \frac{1}{\|\e^1\|^4}}.
\]
From \eqref{chisquare4} it follows that
\[
\E \frac{1}{\|\e^1\|^2}=1, \quad \E \frac{1}{\|\e^1\|^4} \le e^{4 t \|L\|^2}.
\]
Therefore, using \eqref{eqinfdim12} and \eqref{eqinfdim13} we obtain that
\[
\E \|\rho_1-\rho_2\|_{\HC^{1,2}_s}
\le \sqrt t C(t) \|H_1-H_2\|,
\]
which is exactly \eqref{Lindstochnorm12}.

\end{proof}

\section{Stochastic master equations for mean-field interacting particles}
\label{secMFeq}

In \cite{KolQuantLLN} and (for a special case) in \cite{KolQuantMFG}, the author derived the
effective quantum filtering equations for the quantum law of large number limit of interacting particles
under continuous measurement. As for the setting above, these equations can be written either for pure states
as an interacting particle extension of the Belavkin quantum filtering equation representing a new kind
 of stochastic nonlinear Schr\"odinger equation, or for mixed states as stochastic master equations for
 mean-field interacting particles, which can be looked at as an infinite-dimensional complex McKean-Vlasov
 diffusion in the space of positive trace-class operators. These limiting equations provide the forward
 part for the forward-backward system of equations governing the quantum mean-field games.
 In \cite{KolQuantLLN} the well-posedness of the limiting equations for pure states was established.
 Here we aim to establish the well-posedness for the limiting equations for mixed states.

 The stochastic master equations for mean-field interacting particles
 can be formally obtained by adding an interaction term into the Hamiltonian.
 Namely, equation \eqref{Lindstoch} enhanced by mean-field interaction takes the form

 \[
d\ga(t)=-i[H,\ga(t)] \, dt -i[A(\bar \eta(t)), \ga(t)] , dt +\LC_L \ga(t) \, dt
\]
 \begin{equation}
\label{LindstochnewBel}
+(L\ga(t)+\ga(t) L^*) dY(t), \quad \eta(t) =\E (\ga(t)/{\tr}\, \ga (t)).
\end{equation}
Here, $H$, $L$ are as above, $Y(t)$ is a simple $n$-dimensional Ito process,
the expectation $\E$ is with respect to $Y$ and
\[
A: \nu \to A(\nu)
\]
is a linear mapping in the space of bounded linear operators in $\HC$.
In the simplest case (bounded interactions) $A$ satisfies one of the two assumptions:
either $A$ is a
bounded linear mapping $\HC_s^2 \to \HC_s^2$ so that
\begin{equation}
\label{eqinterterm}
\|A(\nu)\|_{\HC^2_s} \le C_A \|\nu\|_{\HC^2_s}
\end{equation}
with a constant $C_A$, or $A$ is a bounded mapping from the trace-class operators to bounded
operators so that
\begin{equation}
\label{eqinterterm1}
\|A(\nu)\| \le C_A \, {\tr}\, |\nu| =C_A\|\nu\|_{\HC^1}
\end{equation}
with a constant $C_A$,

For instance, if $\HC$ is realised as the space $L^2(X, dx)$ of
square integrable functions on some Borel measure space $(X,dx)$,
$A$ satisfying \eqref{eqinterterm} can be given by an integral kernel
$A(x,y;x',y')$ so that, for $\nu\in \HC^2_s$ given by a kernel $\nu(x,y)$,
$A(\nu)$ is the integral operator in $L^2(X,dx)$ with the integral kernel
\[
A(\nu)(x;y)=\int_{X^2} A(x,y;x',y')\nu (y,y') \, dydy'.
\]
In this case
\[
C_A^2=  \int_{X^4} |A(x,y;x',y')|^2 dx dy dx'dy'.
\]
On the other hand, $A$ satisfying \eqref{eqinterterm1} can be given by a
bounded function $A(x,y)$ (interaction potential) so that, for $\nu \in \HC^1_s$ given by a kernel $\nu(x,y)$,
$A(\nu)$ is the operator of multiplication by the function
$\int A(x, y)\nu(y,y) \, dy$.
In this case
\[
C_A= \sup_{x,y} |A(x,y)|.
\]

Notice that \eqref{eqinterterm} implies \eqref{eqinterterm1} for $\nu$ of trace class. In fact,
this is clear for the case of ${\tr} \, \|\nu \| <1$, because in this case
\[
 \|\nu\|_{\HC^2_s} <  \|\nu\|_{\HC^1_s},
\]
and then extends to all trace-class operators by linearity.

Similarly, equation \eqref{Lindstochnorm1} enhanced by a mean-field interaction takes the form

\[
d\rho(t)=-i[H,\rho(t)] \, dt -i[A(\bar \eta(t)), \rho(t)] , dt +\LC_L \rho(t) \, dt
\]
\begin{equation}
\label{eqmainnonlinBel}
+[L\rho(t)+\rho(t) L^*-\rho(t)\, {\tr} \, (L\rho(t)+\rho(t) L^*) ] dB(t),
\quad \eta (t) =\E \rho (t),
\end{equation}
with a $n$-dimensional Brownian motion $B(t)$.

Equations \eqref{eqmainnonlinBel} and \eqref{LindstochnewBel} were derived
rigorously in \cite{KolQuantMFG} and \cite{KolQuantLLN} respectively, as mean-field limit of
continuously observed interacting particle systems. However, in these papers only solutions arising from pure states
(and thus given by the corresponding stochastic nonlinear Schrodinger equations) were discussed.

As in the case without interaction, the same link between equations \eqref{eqmainnonlinBel}
and \eqref{LindstochnewBel} holds. Namely, as one checks by Ito's formula, (i)
if $\ga(t)$ satisfies \eqref{LindstochnewBel}, then $\rho(t)=\ga(t)/{\tr}\, \ga(t)$ satisfies
 \eqref{eqmainnonlinBel}, with $B$ and $Y$ connected via \eqref{outputinnovation}, and (ii) if
 $\rho(t)$ satisfies \eqref{eqmainnonlinBel} and ${\tr} \, \ga(t)$ is chosen as a solution
 to \eqref{eqtrinnov}, then $\ga(t)={\tr}\, \ga (t) \rho(t)$ satisfies  \eqref{LindstochnewBel}.

Equations \eqref{eqmainnonlinBel} and \eqref{LindstochnewBel}
can be considered as infinite-dimensional
complex McKean-Vlasov diffusions on the space of positive trace-class operators
in $\HC$. As above, in order to avoid serious technical issues with Banach-space
valued SDEs, we work with this McKean-Vlasov SDEs as with SDEs in the Hilbert
space $\HC^2_s$ paying attention to the fact that the functional of taking trace
is not continuous in this space.

 \begin{theorem}
\label{mainMFmaster2}
Let $H$ be self-adjoint, $L$ bounded, $\rho_0$ a positive-definite operator of unit trace,
 $B(t)$ a Brownian motion and $A$ satisfy \eqref{eqinterterm} or \eqref{eqinterterm1}.
Then there exists a unique strong solution of equation \eqref{eqmainnonlinBel}
in $\HC^2_s$, with the initial data $\rho_0$
and such that all $\rho(t)$ are positive-definite trace class operators of unit trace.
\end{theorem}

\begin{proof}
We shall deal with the case $H=0$. With this assumption we are not loosing generality,
because otherwise changing $\rho$ to the new variable  $\mu(t)=e^{-iHt}\rho(t) e^{iHt}$
we reduce the story to vanishing $H$ with time-dependent $L$ and $A$. Such time dependence does not
affect the proof in any way. The same strategy was used in Theorem \ref{LindStochLin1} above.


Let $C_{\rho_0}^{1+}([0,T], \HC^2_s)$ be the space of continuous mapping
$\eta: [0,T] \to \HC^2_s$ such that $\eta(0)=\rho_0$ and all $\eta(t)$ are positive
trace class operators of trace not exceeding $1$. It is not difficult to see that
$C_{\rho_0}^{1+}([0,T], \HC^2_s)$ is a complete metric space,
considered as a closed subset of the Banach space of curves in $\HC^2_s$ with the norm
$\sup_{t\in [0,T]} \|\eta(t)\|_{\HC^2_s}$.

Let us define the mapping
\[
\Phi: C_{\rho_0}^{1+}([0,T], \HC^2_s)
\to C_{\rho_0}^{1+}([0,T], \HC^2_s)
\]
by the following rule. To an $\eta \in C_{\rho_0}^{1+}([0,T], \HC^2_s)$
let us assign the solution  of equation
 \[
d r(t)=-i[H,r(t)] \, dt-i[A(\bar \eta(t)), r(t)] , dt
+\LC_L r(t) \, dt
\]
\begin{equation}
\label{eqmainnonlinBela}
+[r(t) L^*+L r(t)-r(t) \, {\tr} (r(t)(L+L^*))] \, dB(t),
\end{equation}
with the initial condition $r_0=\rho_0$ and then define $(\Phi (\eta))(t)= \E r(t)$.
Clearly, $\rho(t)$ is the solution of the Cauchy problem for equation
\eqref{eqmainnonlinBel} with the initial data $\rho_0$ if and only if
$\eta=\E \rho$ is a fixed point of the mapping $\Phi$.

By \eqref{Lindstochnorm12} and \eqref{eqinterterm1},
\[
\| \E \, r_1-\E \, r_2\|_{\HC^1_s}= {\tr}\, |\E \, r_1-\E \, r_2|\le {\tr}\, \E |r_1-r_2|
\]
\[
\le
\sqrt t C(t) \|A(\eta_1)-A(\eta_2)\|\le \sqrt t C(t) C_A \|\eta_1-\eta_2\|_{\HC^1_s}.
\]
Hence, for sufficiently small $t$, the mapping $\Phi$ is a contraction and thus
has a unique fixed point. As usual, existence and uniqueness extends to arbitrary $t$
by iteration.

\end{proof}

\begin{theorem}
\label{mainMFmaster1}
Let $H$ be self-adjoint, $L$ bounded, $\rho_0$ a positive-definite operator of unit trace,
$Y(t)$ a Brownian motion and $A$ satisfy \eqref{eqinterterm} or \eqref{eqinterterm1}.
Then there exists a unique strong solution of equation \eqref{LindstochnewBel}
in $\HC^2_s$, with the initial data $\rho_0$
and such that all $\rho(t)$ are positive-definite trace class operators.
\end{theorem}

\begin{proof}
As above, we can and will choose $H=0$.
The space $C_{\rho_0}^{1+}([0,T], \HC^2_s)$ is also defined as above.

Let us define the mapping
\[
\Phi: C_{\rho_0}^{1+}([0,T], \HC^2_s)
\to C_{\rho_0}^{1+}([0,T], \HC^2_s)
\]
by the following rule. To an $\eta \in C_{\rho_0}^{1+}([0,T], \HC^2_s)$
let us assign the solution  of equation
 \[
d r(t)=-i[H,r(t)] \, dt-i[A(\bar \eta(t)), r(t)] , dt
+\LC_L r(t) \, dt
\]
\begin{equation}
\label{LindstochnewBela}
+[r(t) L^*+L r(t)] \, dY(t),
\end{equation}
with the initial condition $r_0=\rho_0$ and then define
$(\Phi (\eta))(t)= \E \, (r(t)/{\tr} \, r(t))$.
Clearly, $\rho(t)$ is the solution of the Cauchy problem for equation
\eqref{LindstochnewBel} with the initial data $\rho_0$ if and only if
$r(t)$ solves \eqref{LindstochnewBela} and $\eta$ is a fixed point of
the mapping $\Phi$.

As in the proof of the previous theorem we obtain that
\[
\| \E_B \, \frac{r_1}{{\tr}\, r_1}- \E_B \,  \frac{r_2}{{\tr}\, r_2}\|_{\HC^1_s}
\le {\tr}\, \E_B \, \left|\frac{r_1}{{\tr}\, r_1}-\frac{r_2}{{\tr}\, r_2}\right|
\]
\[
\le
\sqrt t C(t) \|A(\eta_1)-A(\eta_2)\|\le \sqrt t C(t) C_A \|\eta_1-\eta_2\|_{\HC^1_s}.
\]
The only problem is that the expectation here is with respect to the Brownian motion $B$
(which we stress by writing $\E_B$)
linked with $Y$ in the usual way, and not with respect to $Y$ itself, as it should be.
However, by Girsanov's theorem,
 expectation with respect to $Y$ and $B$ are linked by a Radon-Nikodyme derivative with all coefficients
 uniformly bounded and thus with bounded second moment. Therefore, from the estimates
 for $\E_B$ we get similar estimates with respect to the expectation $E_Y$,
where $Y$ is a Brownian motion. The proof is again completed by the fixed-point principle.
\end{proof}

\begin{remark} By a slight increase in the length of the calculations one can avoid Girsanov's theorem
and even prove both previous theorems for arbitrary Ito's processes $Y$ and $B$.
\end{remark}

\section{Appendix: some trace inequalities}

\begin{theorem}
If $A$ is a self-adjoint Hilbert-Schmidt operator and $B$ a bounded operator, then
\begin{equation}
\label{eqmyineqtr}
2|{\tr} \, (ABAB^*)|\le {\tr} \, [A^2(BB^*+B^*B)],
\end{equation}
and
\begin{equation}
\label{eqmyineqtr1}
|{\tr} \, (AB AB +AB^*AB^*)|\le {\tr} \, [A^2 (BB^*+B^*B)].
\end{equation}
\end{theorem}

\begin{proof}
by approximation it is reduced to finite-dimensional situation. The diagonalization procedure reduces
the problem to the case when $A$ is a diagonal matrix with real numbers $a_i$ on the diagonal. Then
 \[
 2{\tr} \, (ABAB^*)
=2\sum a_ib_{ij}a_j\bar b_{ij}=\sum a_i a_j (|b_{ij}|^2+|b_{ji}|^2)
\]
\[
=2\sum_i a_i^2 |b_{ii}|^2+2\sum_{i<j}a_ia_j(|b_{ij}|^2+|b_{ji}|^2).
\]
The r.h.s. of \eqref{eqmyineqtr} equals
\[
\sum a_i^2 (|b_{ij}|^2+|b_{ji}|^2)
=2\sum_i a_i^2 |b_{ii}|^2+ \sum_{i < j}(a_i^2+a_j^2)(|b_{ij}|^2+|b_{ji}|^2).
\]
Thus \eqref{eqmyineqtr} holds, because $2|a_ia_j| \le a_i^2 +a_j^2$.

Inequality \eqref{eqmyineqtr1} rewrites as
\[
2\sum_i a_i^2 |Re (b_{ii}^2)|+4\sum_{i<j}|a_ia_j| \, |Re (b_{ij} b_{ji})|
\le 2\sum_i a_i^2 |b_{ii}|^2+ \sum_{i < j}(a_i^2+a_j^2)(|b_{ij}|^2+|b_{ji}|^2),
\]
which easily seen to hold.
\end{proof}

In particular, for self-adjoint $B$ it follows that

\begin{equation}
\label{eqmyineqtr2}
|{\tr} \, (ABAB)|\le {\tr} \, (A^2 B^2).
\end{equation}


\begin{thebibliography}{99}


\bibitem{Armen02Adaptive}
M. A. Armen, J. K. Au, J. K. Stockton, A. C. Doherty and H. Mabuchi. Adaptive homodyne measurement
 of optical phase. Phys. Rev. Lett. 89 (2002), 133602.

\bibitem{BarbRock16}
V. Barbu, M R\"ockner and D. Zhang. Stochastic nonlinear Schrödinger equations.
Nonlinear Anal. 136 (2016), 168 - 194.


\bibitem{BarchBel}
A. Barchielli and V.P. Belavkin. Measurements contunuous in time and a posteriori states in quantum mechanics.
J. Phys A: Math. Gen. 24 (1991), 1495-1514.

\bibitem{BarchHol}
A. Barchielli and A.S. Holevo.
Constructing quantum measurement processes via classical
stochastic calculus. Stochastic Processes and their Applications 58 (1995) 293 - 317.

\bibitem{BarchBook}
A. Barchielli and M. Gregoratti. Quantum Trajectories and Measurements in Continuous Case. The Diffusive Case.
 Lecture Notes Physics, v. 782, Springer Verlag, Berlin, 2009.

\bibitem{Bel87}
 V. P. Belavkin, Nondemolition measurement and control in quantum dynamical systems.
 In: Information Complexity and Control in Quantum Physics.
CISM Courses and Lectures 294, S. Diner and G. Lochak, eds., Springer-Verlag, Vienna, 1987, pp. 331–336.

\bibitem{Bel88}
V.P. Belavkin. Nondemolition stochastic calculus in Fock space and nonlinear
filtering and control in quantum systems. Proceedings XXIV Karpacz winter school
(R. Guelerak and W. Karwowski, eds.), Stochastic methods in mathematics and physics.
 World Scientific, Singapore, 1988, pp. 310 - 324.

\bibitem{Bel92} V.P. Belavkin. Quantum stochastic calculus and quantum nonlinear filtering.
 J. Multivar. Anal. 42 (1992), 171 - 201.

 \bibitem{BelKol}
V.P. Belavkin, V.N. Kolokoltsov. Stochastic
evolution as interaction representation of a boundary value
problem for Dirac type equation. Infinite Dimensional Analysis,
Quantum Probability and Related Fields {\bf 5:1} (2002), 61-92.




\bibitem{BoutHanJamQuantFilt}
L. Bouten, R. Van Handel and M. James. An introduction to quantum filtering.
SIAM J. Control Optim. 46:6 (2007), 2199-2241.


 \bibitem{Bushev06Adaptive}
 P. Bushev et al. Feedback cooling of a singe trapped ion. Phys. Rev. Lett. 96 (2006), 043003.

\bibitem{ChebFagn}
A.M. Chebotarev and F. Fagnola. Sufficient conditions for conservativity of
minimal quantum dynamical semigroups. J. Funct. Anal. 153 (1998), 382 - 404.

\bibitem{ChebQuez}
A.M. Chebotarev, J. Garcia and R. Quezada.
 A priori estimates and existence
theorems for the Lindblad equation with unbounded time-dependent coefficients. In Recent
Trends in Infinite Dimensional Non-Commutative Analysis 1035 (1998), 44–65. Publ. Res.
Inst. Math. Sci., Kokyuroku, Japan.

\bibitem{Fagnola}
F. Fagnola and C. M. Mora.
Stochastic Schr\"odinger equation and applications to Ehrenfest-type theorems.
  ALEA Lat. Am. J. Probab. Math. Stat. 10:1 (2013), 191 - 223.

\bibitem{GreckGenerNonlinSchr}
W. Grecksch and H. Lisei. Stochastic nonlinear equations of Schr\"odinger type.
Stoch. Anal. Appl. 29:4 (2011), 631 - 653.

 \bibitem{Holevo91} A.S. Holevo. Statistical Inference for quantum processes.
 In: Quanum Aspects of Optical communications.
Springer LNP 378 (1991), 127-137, Berlin, Springer.




\bibitem{KolQuantMFGCount}
V. N. Kolokoltsov.
Quantum Mean-Field Games with the Observations of
Counting Type.  Games (2021), 12, 7.

\bibitem{KolQuantLLN}
V. N. Kolokoltsov.
The law of large numbers for quantum stochastic filtering and control of  many particle systems.
 Theoretical and Mathematical Physics 208:1 (2021), 97-121. English translation 208(1), 937-957.

\bibitem{KolQuantMFG}
V. N. Kolokoltsov. Quantum Mean Field Games.
Annals Applied Probability 32:3 (2022), 2254 - 2288.

\bibitem{KolQuantFrac}
V. N. Kolokoltsov.
Continuous time random walks modeling of quantum measurement and
fractional equations of quantum stochastic filtering and control.
Fractional Calculus and Applied Analysis 25 (2022), 128 - 165.

 \bibitem{Pellegrini} C. Pellegrini.
Poisson and Diffusion Approximation of Stochastic Schr\"odinger Equations with Control.
 Ann. Henri Poincar\'e 10:5 (2009), 995–1025.

 \bibitem{Pellegrini10} C. Pellegrini.
Markov chains approximation of jump–diffusion stochastic
master equations.
 Ann. Henri Poincar\'e 46: 4 (2010), 924–948.

\bibitem{Mora13}
C. M. Mora. Regularity of solutions to quantum mastter equattions: a stochastic approach.
The Annals of Probability
41:3B (2013), 1978 - 2012.

\bibitem{MoraRebo}
C. M. Mora and R. Rebolledo.
Basic Properties of Nonlinear Stochastic Schrödinger Equations Driven by Brownian Motions.
 The Annals of Applied Probability 18:2 (2008), 591 - 619.

\bibitem{VanNeerven}
J. van Neerven, M. Veraar and L. Weis (2015).
Stochastic Integration in Banach Spaces – a Survey.
In: Dalang, R., Dozzi, M., Flandoli, F., Russo, F. (eds) Stochastic Analysis:
A Series of Lectures. Progress in Probability, vol 68. Birkhäuser, Basel.
$https://doi.org/10.1007/978-3-0348-0909-2_11$

 \bibitem{WiMilburnBook}
H. M.  Wiseman and G. J.  Milburn.
 Quantum measurement and control. Cambridge Univesity Press, 2010.

\end{thebibliography}
\end{document}